\documentclass[12pt]{article}
\usepackage[total={6in, 9in}]{geometry}
\usepackage{mathrsfs}
\usepackage{t1enc}
\usepackage[latin1]{inputenc}
\usepackage[english]{babel}
\usepackage{amsmath,amsthm,amssymb}
\usepackage{amsfonts}
\usepackage{latexsym}
\usepackage{enumerate}
\usepackage{dsfont}
\newcommand{\ve}{\varepsilon }

\usepackage{graphicx}
\usepackage[natural]{xcolor}
\usepackage{tikz}
\usepackage{tikz-3dplot}
\usetikzlibrary{shapes}
\usetikzlibrary{arrows,decorations.pathmorphing,backgrounds,positioning,fit,petri,automata}
\usetikzlibrary {positioning}

\usepackage[
pdfauthor={},
pdftitle={},
pdfstartview=XYZ,
bookmarks=true,
colorlinks=true,
linkcolor=blue,
urlcolor=blue,
citecolor=blue,
bookmarks=true,
linktocpage=true,
hyperindex=true
]{hyperref}

\theoremstyle{plain}
\newtheorem{thm}{Theorem}[section]

\newtheorem{conj}[thm]{Conjecture}
\newtheorem{claim}[thm]{Claim}

\newtheorem{LEM}[thm]{\textbf{Lemma}}
\theoremstyle{plain}

\theoremstyle{plain}

\theoremstyle{plain}

\begin{document}
\title{Vertex-distinguishing and sum-distinguishing edge coloring of regular graphs}
\author{Yuping Gao\footnote{Lanzhou University, School of Mathematics and Statistics, Lanzhou 730000, China.}
	\qquad
	Songling Shan\footnote{Auburn University, Department of Mathematics and Statistics, Auburn, AL 36849, USA.}\qquad Guanghui Wang\footnote{Shandong University, School of Mathematics, Jinan 250100, China.}
}

\date{\today}
\maketitle
\begin{abstract} Given an integer $k\ge1$,
an edge-$k$-coloring of a graph $G$ is an assignment of $k$ colors $1,\ldots,k$ to the edges of $G$ such that no two adjacent edges receive the same color. A vertex-distinguishing (resp. sum-distinguishing) edge-$k$-coloring of $G$ is an edge-$k$-coloring  such that for any two distinct vertices $u$ and $v$, the set (resp. sum) of  colors taken  from all the edges incident with $u$ is different from that taken from all the edges incident with $v$. The vertex-distinguishing chromatic index (resp. sum-distinguishing chromatic index),
denoted $\chi'_{vd}(G)$ (resp.  $\chi'_{sd}(G)$), is  the smallest value $k$ such that $G$ has  a vertex-distinguishing-edge-$k$-coloring
(resp. sum-distinguishing-edge-$k$-coloring).
Let $G$ be a   $d$-regular graph  on $n$ vertices, where $n$ is even and sufficiently large. We show that
$\chi'_{vd}(G) =d+2$ if  $d$ is arbitrarily  close to $n/2$ from above, and $\chi'_{sd}(G) =d+2$ if $d\ge \frac{2n}{3}$.
Our  first  result  strengthens a result of Balister et al. in 2004 for such class of regular graphs, and
our second result constitutes a significant advancement in the field of sum-distinguishing edge coloring.
 To achieve these results, we introduce novel edge coloring results
 which  may be of independent interest.

\medskip

\end{abstract}

\section{Introduction}

Graphs in this paper are simple and finite.
Let $G$ be a graph.  The vertex set and edge set of $G$ are denoted by $V(G)$ and $E(G)$, respectively.
Let  $v\in V(G)$ and $S\subseteq V(G)$.  Then $N_G(v)$ is the set of neighbors of $v$ in $G$,
 $d_G(v):=|N_G(v)|$ is the degree of $v$ in $G$, and $d_S(v)=|N_G(v)\cap S|$. The subgraph of $G$ induced on  $S $  is   denoted by $G[S]$.
 For any two disjoint subsets $A,B\subseteq V(G)$, we use $E_G(A,B)$ to denote the set of edges in $G$ with one endvertex in $A$ and the other endvertex in $B$.  For $F\subseteq E(G)$, $G[F]$ is the subgraph of $G$ induced on $F$. For $e\in E(G)$, $G-e$
 is obtained from $G$ by deleting $e$.
 For two integers $p$ and $q$, let $[p,q]=\{i\in \mathbb{Z}\mid p\leq i\leq q\}$.

Given an integer $k\ge 0$,
 an \emph{edge-$k$-coloring} of a graph $G$ is an assignment of $k$ colors from $[1,k]$ to the edges of $G$ such that no two adjacent edges receive the same color.  A \emph{color class} of the coloring is the  set of edges assigned with the same color.
 The \emph{chromatic index} of  $G$, denoted   $\chi'(G)$, is the minimum integer $k$ such that $G$ admits an edge-$k$-coloring. 	In 1960s, Gupta~\cite{Gupta-67}  and  Vizing~\cite{Vizing-2-classes}  independently proved
	that for all  graphs $G$,  $\Delta(G) \le \chi'(G) \le \Delta(G)+1$.  Various edge coloring variants have been studied, often imposing additional constraints. For example, acyclic edge coloring and list edge coloring. In this paper, we focus on two specific variants: vertex-distinguishing edge coloring and sum-distinguishing edge coloring.

 Given an edge-$k$-coloring $\varphi$, we can deduce two distinct  vertex labelings of $G$ based on $\varphi$ as follows. For
 each $v\in V(G)$, let $s_\varphi(v)=\{\varphi(vu)| u\in N_G(v)\}$ and $\omega_{\varphi}(v)=\sum_{u\in N_G(v)}\varphi(vu)$. We call $\varphi$ a \emph{vertex-distinguishing edge-$k$-coloring (vd-edge-$k$-coloring)}  if $s_\varphi(u) \ne s_\varphi(v)$ for any distinct $u,v\in V(G)$, and we call
 $\varphi$ a
 \emph{sum-distinguishing edge-$k$-coloring (sd-edge-$k$-coloring)} if $\omega_\varphi(u) \ne \omega_\varphi(v)$ for any distinct $u,v\in V(G)$.  The \emph{vertex-distinguishing chromatic index} (resp. \emph{sum-distinguishing chromatic index}),
denoted $\chi'_{vd}(G)$ (resp.  $\chi'_{sd}(G)$), is  the smallest value $k$ such that $G$ has  a vd-edge-$k$-coloring
(resp. sd-edge-$k$-coloring).  It is clear that $\chi'(G) \le \chi'_{vd}(G) \le \chi'_{sd}(G)$.

Vertex-distinguishing edge coloring was introduced independently by Burris and Schelp in~\cite{BS1997}, and by Aigner et al. in~\cite{ATT1992}.
A graph $G$ has a vd-edge-$k$-coloring or sd-edge-$k$-coloring if and only if $G$ contains at most one isolated vertex and no isolated edges.  In~\cite{BS1997},  the vertex-distinguishing chromatic index was established  for complete graphs, complete bipartite graphs, paths and cycles. Burris and Schelp  also proposed the following conjecture.

\begin{conj}[Burris and Schelp~\cite{BS1997}]\label{conj1-3}  Let $G$ be a  graph  with at most one isolated vertex and with no isolated edge. Let $\pi(G)=\min\{k\in \mathbb{N} \mid \binom{k}{d}\geq n_{d}\}$ for all $\delta(G)\leq d\leq \Delta(G)$, where $n_{d}$ is the number of vertices with degree $d$ in $G$.  Then $\chi'_{vd}(G) \in \{\pi(G), \pi(G)+1\}$.
\end{conj}

Conjecture~\ref{conj1-3} was confirmed for graphs $G$ with $\Delta(G)=2$~\cite{BBS2002}. In~\cite{BHLW1999}, the authors proved that $\chi'_{vd}(G)\leq |V(G)|+1$ for any  graph $G$ that has at most one isolated vertex but no isolated edge. Later, they proved that $\chi'_{vd}(G)\leq \Delta(G)+5$ for any graph $G$ with $\delta(G)>\frac{|V(G)|}{3}$~\cite{BHLW2001}. The  minimum degree condition above was improved to a minimum degree sum condition in~\cite{LL2010}. Conjecture~\ref{conj1-3} was also confirmed for graphs $G$ with $\Delta(G)\geq \sqrt{2|V(G)|}+4$ and $\delta(G)\geq 5$~\cite{BKLS2004}. In this paper we will give the exact value of vertex-distinguishing edge chromatic index for large regular graphs with even order and it strengthens the result in~\cite{BKLS2004} for such classes of graphs. Furthermore,
 we make some advancement in the field of sum-distinguishing edge coloring.


\begin{thm}\label{thm}  The following statements hold.
\begin{enumerate} \item[{\rm (i)}] For any $0<\ve <1$, there exists $n_0\in \mathbb{N}$ such that
if $G$ is  a $d$-regular graph on $n\geq n_0$ vertices with $d\geq \frac{(1+\varepsilon)n}{2}$
and  $n$ even,    then $\chi'_{vd}(G)=d+2$.
\item[{\rm (ii)}] There exists $n_0\in \mathbb{N}$ such that
if  $G$ is  a $d$-regular graph on $n\geq n_0$ vertices with $d\geq \frac{2n}{3}$
and  $n$ even,    then $\chi'_{sd}(G)=d+2$.
\end{enumerate}
\end{thm}

To the best of our knowledge, Theorem~\ref{thm}(ii) represents the first major breakthrough in the field of sum-distinguishing edge coloring, which has historically faced significant challenges.
The work in~\cite{MSJ2012} serves as a foundational reference in this area, where
the authors   calculated both the vertex-distinguishing chromatic index and the sum-distinguishing chromatic index for complete graphs.

\begin{thm}[Mah\'{e}o and Sacl\'{e}~\cite{MSJ2012}]\label{prop2} If $n\geq 3$, then
\[\chi'_{vd}(K_n)=\chi'_{sd}(K_n)=\left\{\begin{array}{cc}
                         n & \mbox{if}\ n\ \mbox{is\ odd};\\
                         n+1& \mbox{if}\ n\ \mbox{is\ even}.
                       \end{array}\right.
\]
\end{thm}

The remainder of this paper is organized as follows. We provide necessary   preliminaries in the next section, and then we prove Theorem~\ref{thm} in Section 3.

\section{Preliminaries}

Let $G$ be a graph.
A \emph{matching}  $M$ in  $G$ is a set of vertex-disjoint edges.  The set of vertices \emph{covered} or \emph{saturated} by $M$ is denoted by $V(M)$.
We  say  that $M$ is \emph{perfect}    if $V(M)=V(G)$. The following result can be easily proved by applying Hall's  Theorem, and one
can also find a short proof in Lemma 9 of~\cite{V2013}.

\begin{LEM}\label{lem:PM} Let $G[X, Y ]$ be a bipartite graph with $|X| = |Y | = n$. If $\delta(G)\geq n/2$, then $G$ has
a perfect matching.
\end{LEM}

Given an integer $k\ge 1$,
a \emph{$k$-coloring} of  $G$ is an assignment of $k$ colors from $[1,k]$ to the vertices of $G$ such that adjacent vertices receive different colors. Given a $k$-coloring $\varphi$, and given any $i\in [1,k]$, we call $V_i =\{v\in V (G):\varphi(v) =i\}$ the $i$th color class of $\varphi$; if $|V_i|,|V_j|$ differ by at most one for any $i,j\in[1,k]$, then we say that $\varphi$ is \emph{equitable}.
Hajnal and Szemer\'edi~\cite{HS1970} in 1970 proved the following classic result.

\begin{thm}[Hajnal and Szemer\'edi~\cite{HS1970}]\label{thm-equitable} Let $G$ be a graph and let $k\geq \Delta(G) + 1$ be any
integer. Then $G$ has an equitable $k$-coloring.
\end{thm}

The remainder of this section focuses on establishing key results related to edge coloring graphs with precolored edges.

\begin{thm}[K\"{o}nig~\cite{K1916}]\label{konig}
	Every bipartite graph  $G$ satisfies $\chi'(G)=\Delta(G)$.
\end{thm}

An edge-$k$-coloring of a graph $G$ is \emph{equitable} if each color class  has size $\lfloor |E(G)|/k \rfloor$ or  $\lceil |E(G)|/k \rceil$.
The following theorem by McDiarmid from 1972 states that every graph has an equitable edge coloring  as long as at least $\chi'(G)$ colors are given.

\begin{thm}[McDiarmid~\cite{M1972}]\label{lem:equa-edge-coloring}
 	Let $G$ be a graph and  $k\ge \chi'(G)$ be an integer.  Then $G$ has an equitable edge-$k$-coloring.
\end{thm}

An equitable edge coloring can always be achieved by modifying an arbitrary given edge coloring.
We will establish a result analogous to Theorem~\ref{lem:equa-edge-coloring} by modifying a given edge coloring, while preserving the colors of specific edges. To state our result, we provide some definitions.

Let $G$ be a graph, $k\ge 0$ be an integer, and   $\varphi$ be an edge-$k$-coloring of $G$.  For  $v\in V (G)$, denote by $\overline{\varphi}(v)$ the set of colors \emph{missing} at $v$, that is, the set of all colors in $[1,k]$ that are not used on any edge incident to $v$ under $\varphi$. For $i\in [1,k]$,
  denote by $\overline{\varphi}^{-1}(i)$ the set of all vertices in $G$ which are missing color $i$ under $\varphi$.
  For a vertex set $X\subseteq V (G)$, define $\overline{\varphi}(X) =\cup_{ v\in X} \overline{\varphi}(v)$ to be the set of missing colors of $X$. Given any pair of distinct colors $\alpha,\beta\in [1,k]$,  each component of
  the subgraph of $G$  induced on all edges colored by $\alpha$ or $\beta$ is either a cycle or a path.
  Each of such component is called an \textit{$(\alpha,\beta)$-chain of $G$ with respect to $\varphi$}.
  Interchanging  $\alpha$ and $\beta$
on an $(\alpha,\beta)$-chain $C$ of $G$ gives a new edge-$k$-coloring, which is denoted by
$\varphi/C$.
This operation  is called a \emph{Kempe change}.

If an $(\alpha,\beta)$-chain  $P$ is a path with one endvertex as $x$, we also denote it by $P_x(\alpha,\beta,\varphi)$, and we just write $P_x(\alpha,\beta)$ if $\varphi$ is understood.   For a vertex $u$ and an edge $uv$ contained in $P_x(\alpha,\beta,\varphi)$,
we write
{$\mathit {u\in P_x(\alpha,\beta, \varphi)}$} and  {$\mathit {uv\in P_x(\alpha,\beta, \varphi)}$}.

We are now ready to state and prove the following result. The case $m=1$ of the result was proved in~\cite{DMS2024}.

\begin{LEM}\label{PartialEC}  Let $G$ be a graph and $\varphi_0$ be an edge-$k$-coloring of $G$. Suppose $E_0\subseteq E(G)$
such that each color in $[1,k]$ is used on at most $m$ edges of $E_0$, where $m\ge 0$ is an integer.
 Then $G$  has an edge-$k$-coloring $\varphi$  such that
\[||\overline{\varphi}^{-1}(i)|-|\overline{\varphi}^{-1}(j)||\leq 4m+1\]
for all $i,j\in[1,k]$
and that $\varphi(e)=\varphi_0(e)$ for any  $e\in E_0$.
\end{LEM}

\begin{proof} We call an edge-$k$-coloring of $G$ \emph{good} if each color is used on at most $m$ distinct edges of $E_0$. Among all good edge-$k$-colorings of $G$, we choose $\varphi$ so that  $\varphi(e)=\varphi_0(e)$ for any edge $e\in E_0$ and
\[g_{\varphi}:= \max\{||\overline{\varphi}^{-1}(i)|-|\overline{\varphi}^{-1}(j)||\, \mid  i,j\in [1,k]\}
\]
is minimum. If $g_{\varphi}\leq 4m+1$, then we are done, so suppose not. Thus there are colors $\alpha,\beta \in[1,k]$ such that
\begin{equation}\label{eq1}||\overline{\varphi}^{-1}(\alpha)|-|\overline{\varphi}^{-1}(\beta)||=g_{\varphi}\geq 4m+2.
\end{equation}
We may further assume that $\varphi$ has been chosen so as to minimize the number of pairs $\alpha,\beta$ satisfying~\eqref{eq1}; let $h_{\varphi}$ be the number of these pairs.

Without loss of generality, suppose that $|\overline{\varphi}^{-1}(\alpha)|-|\overline{\varphi}^{-1}(\beta)|\geq 4m+2$. There are at most
$2m$ edges from $E_0$ that are colored from $\{\alpha,\beta\}$. Thus there are at most $4m$ vertices from $\overline{\varphi}^{-1}(\alpha)$ that are endvertices of $2m$  $(\alpha,\beta)$-chains (path chains) that each contains an edge of $E_0$.
At most $|\overline{\varphi}^{-1}(\beta)|$ vertices from $\overline{\varphi}^{-1}(\alpha)$ are endvertices of $|\overline{\varphi}^{-1}(\beta)|$ $(\alpha,\beta)$-chains involving a vertex from $\overline{\varphi}^{-1}(\beta)$. Thus there exist distinct vertices $x, y\in \overline{\varphi}^{-1}(\alpha)$ occurring in one common $(\alpha,\beta)$-chain that contains no edge of $E_0$. Let $P$ be the $(\alpha,\beta)$-chain with endvertices $x, y$, and let $\psi=\varphi/ P$. As $P$ does not contain any edges of $E_0$,
we have $\varphi(e)=\psi(e)$ for each $e\in E_0$.
Hence $\psi$ is a good
edge-$k$-coloring of $G$. We claim that $g_{\psi}\leq g_{\varphi}$ and $h_{\psi}< h_{\varphi}$, which yields a contradiction to our choice of $\varphi$. By our choice of $P$, we know that $|\overline{\psi}^{-1}(\alpha)|=|\overline{\varphi}^{-1}(\alpha)|-2$ and $|\overline{\psi}^{-1}(\beta)|=|\overline{\varphi}^{-1}(\beta)|+ 2$.
Given \eqref{eq1} and the definition of $g_{\varphi}$, for any other color $\gamma\in [1,k]\setminus \{\alpha,\beta\}$, we have
\[|\overline{\varphi}^{-1}(\alpha)|\geq |\overline{\varphi}^{-1}(\gamma)|\geq |\overline{\varphi}^{-1}(\beta)|.\]
But now, for any pair of colors $\{\alpha',\beta'\}$ satisfying \eqref{eq1}, we have
\[||\overline{\psi}^{-1}(\alpha')|-|\overline{\psi}^{-1}(\beta')||=\left\{\begin{array}{cc}
                                                      g_{\varphi} & \mbox{if}\ \{\alpha',\beta'\}\cap \{\alpha,\beta\}=\emptyset,\\
                                                      g_{\varphi}-2 & \mbox{if}\ \alpha'= \alpha\ \mbox{and}\ \beta'\neq \beta,\\
                                                       g_{\varphi}-2 & \mbox{if}\ \alpha'\neq \alpha\ \mbox{and}\ \beta'= \beta, \\
                                                       g_{\varphi}-4 & \mbox{if}\ \alpha'= \alpha\ \mbox{and}\ \beta'= \beta.
                                                    \end{array}
\right.\]
Hence we have $g_{\psi} \leq g_{\varphi}$ and $h_{\psi}< h_{\varphi}$, a contradiction to our choice of $\varphi$.
\end{proof}

Our next result will be proved using the concept of multifans.
The technique of multifan is widely used in edge coloring problems and was firstly appeared in the book of Stiebitz et al. \cite{StiebSTF-Book}  as a generalization of Vizing-fan.  Let $\varphi$ be an edge-$k$-coloring of $G-e_0$ for some edge $e_0\in E(G)$, and $r\in V(G)$
be a vertex.
A \emph{multifan} centered at $r$ with respect to $\varphi$ is a sequence $F_{\varphi}(r, s_0 : s_p) = (r, e_0, s_0, e_1, s_1,\ldots, e_p, s_p)$ with $p \geq 0$ consisting of distinct edges $e_0,\ldots, e_p$ with $e_i = rs_i$ for all $i\in [0,p]$, where $e_0$ is left uncolored
under $\varphi$, and for every edge $e_i$ with $i \in [1,p]$, there exists $j\in [0,i-1]$ such that $\varphi(e_i)\in \overline{\varphi}(s_j)$. Note that if $F_{\varphi}(r, s_0 : s_{p})$ is a multifan, then for any integer $p^{*}\in [0,p],F_{\varphi}(r, s_0 : s_{p^*})$ is also a multifan. Given $0\leq t\leq p$, a subsequence $(s_0 = s_{\ell_0},s_{\ell_1},\ldots,s_{\ell_t})$ with the property
 that $\varphi(e_{\ell_i}) =\alpha\in\overline{\varphi}(s_{\ell_{i-1}})$ for each $i\in [1,t]$, is called a \emph{linear sequence} of $F_{\varphi}(r, s_0 : s_p)$. Given such a linear sequence, and given $h \in \{1, \ldots, t\}$, we \emph{shift from $s_{\ell_h}$ to $s_0$} by recoloring edge $e_{\ell_{i-1}}$ with $\varphi(e_{\ell_{i}})$ for all $i\in\{1, \ldots, h\}$.
Note that the resulting edge coloring is an edge-$k$-coloring where $e_0$ is colored and $e_{\ell_h}$ is not.

\begin{LEM}\label{lem2.2}
Let $G$ be a graph, $Q\subseteq G$ and $k,c,m\geq 1$ be integers.
Suppose that $\Delta(Q) \le c$ and that  $Q$  has an edge-$k$-coloring  $\varphi_0$ such that each color $i\in [1,k]$ is used on at most $m$ edges of $Q$.
Then we can extend the coloring $\varphi_0$ of $Q$ to an edge coloring of $G$ using
at most $\max\{k,\Delta(G)+4cm-1\}$ colors.
\end{LEM}

\begin{proof} Let $\Delta = \Delta(G)$, $\ell=\max\{k,\Delta(G)+4cm-1\}$ and  the color set be  $[1,\ell]$.
We extend $\varphi_0$ to a partial edge-$\ell$-coloring $\varphi$ of $G$ by preserving colors of $E(Q)$ such that $\varphi$ assigns colors to as many edges as possible.

If every edge in $G$ is colored under $\varphi$ then we are done. Now assume that there is some edge $e_0=uv$ that is uncolored by $\varphi$.  Since $e_0$ is uncolored, we have  $e_0\not\in E(Q)$. Let $F$ be a maximal multifan centered at $u$ with respect to $e_0$ and $\varphi$.

Since $\varphi$ has $\ell\geq\Delta+4cm-1$ colors (and $uv$ is uncolored), we know that $|\overline{\varphi}(v)|\geq 4cm$, and each color in $\overline{\varphi}(v)$ must be present at $u$ (otherwise  $\varphi$ could be extended by adding $e_0 = uv$). So there exist linear sequences of $F$ containing no edge of $E(Q)$. Let $F_1 \subseteq F$ be a multifan induced by all the linear sequences of $F$ that don't contain edges of $E(Q)$. Denote by $F_1 = (u, uw_0, w_0, uw_1, w_1, \ldots, uw_t, w_t)$, where $w_0 = v$.

\begin{claim}\label{claim1}Suppose $F_2\subseteq F_1$ corresponds to the union of any number of linear sequences of $F$. Then $\overline{\varphi}(u)\cap \overline{\varphi}(w_i) = \emptyset$ for any $w_i\in V (F_2)$.
\end{claim}

\begin{proof}  Suppose to the contrary that  $\overline{\varphi}(u)\cap \overline{\varphi}(w_i)\neq \emptyset$ for some $i$.
Then there exists a linear sequence $(u, w_0, w_{i_1},\ldots, w_{i_s})$ of $F_2$ for which $w_{i_s}= w_i$. Shift in $(u, w_0, w_{i_1},\ldots, w_{i_s})$ from $w_{i_s}$
to $w_0$ and then color the edge $uw_i$ by a color in $\overline{\varphi}(u)\cap \overline{\varphi}(w_i)$. The resulting partial edge-$\ell$-coloring $\varphi'$ of $G$ is an extension of $\varphi_0$, since we excluded the edges of $E(Q)$ from $F_1$ (and $F_2$). However,  $\varphi'$ colors more edges than $\varphi$ does, a contradiction to the choice of $\varphi$.
\end{proof}

\begin{claim}\label{claim2} For any color $\gamma\in [1,\ell]$, there are at most $4m-1$ vertices from $\{w_0,\ldots, w_t\}$ that are all missing $\gamma$ under $\varphi$.
\end{claim}

\begin{proof}Suppose to the contrary that there exists such a color $\gamma$ that is missed on at least $4m$ vertices from $\{w_0,\ldots, w_t\}$. Let $y_1, y_2, \ldots, y_{4m}$ be $4m$ distinct vertices of $\{w_0,\ldots, w_t\}$ all of which are missing $\gamma$ under $\varphi$. Suppose, without loss of generality, that $y_1$ is the first vertex in the order $w_0,\ldots, w_t$ for which $\gamma\in \overline{\varphi}(y_1)$.
In particular, this implies that $F_1(u, w_0 : y_1)$ does not contain any edge that is colored by $\gamma$. Let $\alpha\in \overline{\varphi}(u)$, then $\alpha\neq\gamma$ by Claim~\ref{claim1}. Consider the set $\mathcal{P}$ of  $(\alpha, \gamma)$-chains
starting at $u, y_1,\ldots, y_{4m}$, respectively. At most $m$ edges of $E(Q)$ are colored by $\alpha$, and at most $m$ edges of $E(Q)$ are colored by $\gamma$. So there are at most $4m$ vertices in $\{u,w_0,y_1,\ldots,y_{4m}\}$ which can be endvertices of chains in $\mathcal{P}$. It follows that there is at least one chain $P \in \mathcal{P}$ not containing any edges from $E(Q)$. Let $\varphi'= \varphi/ P$ and consider two different cases according to whether or not $u$ is an endvertex of $P$.

{\bf \noindent Case 1}: The vertex $u$ is an endvertex of $P$.

Suppose first that the other endvertex of $P$ is $y_1$. Then in $\varphi'$, $u$ is missing $\gamma$. So in particular no edges colored with $\gamma$ appear in $F_1$. If there were any before, then they must have been recolored to $\alpha$. However since these edges all occur after $y_1$ in the order of $F_1$, and $y_1$ is missing $\alpha$ in $\varphi'$, $F_1$ is still a multifan in $G$ under $\varphi'$. However now Claim \ref{claim1} is not satisfied for $\varphi'$,   a contradiction to our choice of $\varphi$.

Now assume that the other endvertex of $P$ is not $y_1$. Note that $F_1(u, w_0 : y_1)$ is a multifan with respect to $\varphi'$ containing no edges of $E(Q)$. But now the color $\gamma$ is missing at both $u$ and $y_1$ under $\varphi'$, so Claim \ref{claim1} is not satisfied for $\varphi'$,   a contradiction to our
choice of $\varphi$.

{\bf \noindent Case 2}: The vertex $u$ is not an endvertex of $P$.

In this case $P$ does not contain any edges of $F_1$.

First suppose that $y_1$ is not an endvertex of $P$. Then $F_1$ is still a multifan with respect to $\varphi'$ containing no edges of $E(Q)$. However, the color $\alpha$ is missing at $u$ and at least one other vertex from $\{y_2,\ldots, y_{4m}\}$ under $\varphi'$. This means that Claim \ref{claim1} is not satisfied for $\varphi'$, leading to a contradiction to  our choice of $\varphi$.

Now assume that $y_1$ is an endvertex of $P$. Here, at least $F_1(u, w_0 : y_1)$ is a multifan with respect to $\varphi'$ containing no edges of $E(Q)$. But now the color $\alpha$ is missing at both $u$ and $y_1$ under $\varphi'$, so Claim \ref{claim1} is not satisfied for $\varphi'$,   a contradiction to  our choice
of $\varphi$.
\end{proof}

Claims \ref{claim1} and \ref{claim2} imply  that
\begin{equation}\label{eq0}
|\overline{\varphi}(V (F_1))| \geq |\overline{\varphi}(u)| + \frac{1}{4m-1}\sum\limits_{i=0}^{t}
|\overline{\varphi}(w_i)|.
\end{equation}
Let $G'$ be the graph induced on the set of  all the colored edges of $G$.  Then $|\overline{\varphi}(z)|= \ell-d_{G'}(z)\geq\Delta+4cm-1- d_{G'}(z)$ for all $z \in V (G)$. Since $d_{G'}(z) \leq \Delta$, we know that $|\overline{\varphi}(z)|\geq 4cm-1$ for
any $z$. Since the edge $uw_0$ is uncolored, we further know that $|\overline{\varphi}(w_0)|\geq 4cm$. So from~\eqref{eq0},
we get
\begin{eqnarray}
|\overline{\varphi}(V (F_1))| &\geq& (\Delta+4cm-1-d_{G'}(u))+\frac{4cm}{4m-1} +\frac{1}{4m-1} \cdot (4cm-1)t \nonumber\\
&>&\Delta+4cm+c-1+ct-d_{G'}(u).\label{eq01}
\end{eqnarray}

On the other hand, there are exactly $t$ colored edges in $F_1$. So there are exactly $d_{G'}(u)-t$ colored edges incident to $u$ that are not included in $F_1$. Furthermore, at most $c$ of these excluded colored edges could be in $E(Q)$ as $d_{Q}(u)\leq c$. By  the definition of $F_1$, none of the above excluded
colors are in $\overline{\varphi}(V (F_1))$. Hence we have found a set of at least $d_{G'}(u) -t- c$ colors from $[1,\Delta+4cm-1]$ that are not in $\overline{\varphi}(V (F_1))$. So
\[|\overline{\varphi}(V (F_1))| \leq (\Delta + 4cm-1) - (d_{G'}(u)- t -c) = \Delta +4cm+ t+c-1- d_{G'}(u),\]
contradicting \eqref{eq01}.
\end{proof}

As the number of vertices covered by a matching is always even, we have the following   lemma.

\begin{LEM}\label{Parity} Let $G$ be a graph and $\varphi$ be an edge-$k$-coloring of $G$ for some $k\geq \Delta(G)$. Then $|\overline{\varphi}^{-1}(i)|\equiv |V(G)|\ ({\rm mod}\ 2)$ for every color $i\in[1,k]$.
\end{LEM}

We will also need the following partition lemma by the second author.

\begin{LEM}[Shan~\cite{S2022}]\label{lem:partition}
There exists a positive integer $n_0$ such that for all $n\geq n_0$ the following holds. Let $G$ be a graph on $2n$ vertices, and $N = \{x_1, y_1,\ldots, x_t, y_t\}\subseteq V(G)$, where $1\leq t\leq n$ is an integer.
Then $V(G)$ can be partitioned into two parts $A$ and
$B$ satisfying the properties below:
\begin{enumerate}
\item [{\rm (i)}] $|A| = |B|$;
\item [{\rm (ii)}] $|A\cap \{x_i, y_i\}| = 1$ for each $i\in [1,t]$;
\item [{\rm (iii)}] $|d_A(v)-d_B(v)| \leq n^{\frac{2}{3}}$ for each $v\in V (G)$.
\end{enumerate}

Furthermore, one such partition can be constructed in $O(2n^3 {\rm log}_2(2n^3))$-time.
\end{LEM}

Combinatorial design theory traces its origins to statistical theory of experimental design but also to recreational mathematics of the 19th century and to geometry. A \emph{balanced incomplete block design} (BIBD) with parameters $(\nu, b, r, k, \lambda)$
is an ordered pair $(\Omega, \mathcal{B})$ where $\Omega$ is a finite $\nu$-element set of elements or points, $\mathcal{B}=\{B_1,B_2,\ldots,B_{b}\}$ is a family of $k$-element subsets of $\Omega$, called \emph{blocks} such that every
point is contained in exactly $r$ blocks, every 2-subset of $\Omega$ is contained
in exactly $\lambda$ blocks, and $k<\nu$. It is well known that there exists a $(9,12, 4, 3, 1)$-BIBD. We will need
this in the proof of Theorem~\ref{thm}(i).

\section{Proof of Theorem~\ref{thm}}

We will prove Theorem~\ref{thm}(i) and (ii) together. The proof  requires different discussion in the initial and last steps,
but the rest parts are basically the same.

 We choose $n_0$ to be at least  the $n_0$ specified in Lemma~\ref{lem:partition} such that $1/n_0 \ll \ve \ll 1$.
Let $G$ be a $d$-regular graph on $n\ge n_0$ vertices, where $n$ is even.
We are done by Proposition~\ref{prop2} if $G=K_n$. Thus assume that $G$ is not $K_n$. Then we have $\chi'_{vd}(G)\geq \pi(G)\geq d+2$.  To prove the theorem, it suffices to show that  $\chi'_{vd}(G) \le d+2$ when $d\ge \frac{(1+\varepsilon)n}{2}$ and $\chi'_{sd}(G) \le d+2$ when $d\ge \frac{2n}{3}$
as $\chi'_{vd}(G) \le \chi'_{sd}(G)$.

 For   Theorem~\ref{thm}(i), we will add to $G$ the edges of a   2-regular graph $Q$ on $V(G)$
and precolor the new added edges from $Q$ (multiple edges may occur in the resulting multigraph $G^*$).  We  then extend the precoloring on the edges of $Q$ to an edge-$(d+2)$-coloring  $\varphi$ of $G^*$. By making the initial
edge coloring on edges from  $Q$ to be vertex-distinguishing,   the restriction of $\varphi$ on $G$ gives a vd-edge-$(d+2)$-coloring of $G$
as $G^*$ is $(d+2)$-regular.
For Theorem~\ref{thm}(ii), we will find in $G$ a spanning subgraph $Q$ with maximum degree at most 3 and precolor
its edges.   We  then extend the precoloring  of $Q$ to $G$ using $d$ colors.
As $G$ is $d$-regular, the sums of colors on edges incident
with all vertices are equal after the  extension.
In the last step,  we modify colors on some edges of $Q$ by using   two new colors.
By carefully designing the initial precoloring of $Q$'s edges and recoloring some of its edges using two new colors in the last step, we can ensure that the resulting coloring is sum-distinguishing.  We now provide details for
the construction of the colorings.

\begin{center}
    { \bf Initial Step: Dentition of $Q$ and its pre-edge-coloring $\varphi_0$}
\end{center}

Let $n=3q+r$, where  $q\in \mathbb{N}$ and $r\in[0,2]$.

For Theorem~\ref{thm}(i), we let $Q$ to be a 2-regular graph on $V(G)$ defined as follows:
\begin{itemize}
\item If $r=0$, then $Q$ is the vertex-disjoint union of $q$ copies of 3-cycles;
\item If $r=1$, then $Q$ is the vertex-disjoint union of $q-1$ copies of 3-cycles and one copy of a 4-cycle;
\item If $r=2$, then $Q$ is the vertex-disjoint union of $q-1$ copies of 3-cycles and one copy of a 5-cycle.
\end{itemize}
Let $G^{*}$ be the  multigraph obtained from $G$ by adding all edges of $E(Q)$. Then  $G^*$ is $(d+2)$-regular.

Let $c\geq 10$ be the smallest constant such that $\lfloor\frac{n}{4}\rfloor+c=3t$ for some $t\in \mathbb{N}$. We define $t$ pairwise nonintersecting sets as $T_i=\{3i-2,3i-1,3i\}$ for $i\in[1,t]$. For each $j\in[1,\lfloor\frac{t}{3}\rfloor]$, there exists a $(9,12,4,3,1)$-BIBD  on the set
\[T_{3j-2}\cup T_{3j-1}\cup T_{3j}=\{9j-8,9j-7,9j-6,9j-5,9j-4,9j-3,9j-2,9j-1,9j\}\]  as follows:
\begin{equation} \nonumber
\begin{array}{ll}
    B^1_{j}=\{9j-8,9j-7,9j-6\}, & B^2_{j}=\{9j-5,9j-4,9j-3\}, \\
    B^3_{j}=\{9j-2,9j-1,9j\}, & B^4_{j}=\{9j-8,9j-5,9j-2\}, \\
    B^5_{j}=\{9j-7,9j-4,9j-1\},& B^6_{j}=\{9j-6,9j-3,9j\},\\
    B^7_{j}=\{9j-8,9j-4,9j\},& B^8_{j}=\{9j-7,9j-3,9j-2\}, \\
    B^9_{j}=\{9j-6,9j-5,9j-1\},& B^{10}_{j}=\{9j-8,9j-3,9j-1\}, \\
    B^{11}_{j}=\{9j-7,9j-5,9j\}, & B^{12}_{j}=\{9j-6,9j-4,9j-2\}.
\end{array}
\end{equation}

Thus we created in total
\[\begin{aligned}12\times \left\lfloor\frac{t}{3}\right\rfloor &\geq 12\left(\frac{t}{3}-1\right)=4\left(\frac{1}{3}\left\lfloor\frac{n}{4}\right\rfloor+\frac{c}{3}\right)-12\geq \frac{4}{3}\left(\frac{n}{4}-1\right)+\frac{4c}{3}-12\\&\geq \frac{n}{3}+\frac{4c}{3}-\frac{40}{3}\geq \frac{n}{3}
\end{aligned}\]
3-sets.  We let $B_1, \ldots, B_{\lfloor\frac{n}{3}\rfloor}$ be $\lfloor\frac{n}{3}\rfloor$
of these 3-sets.

For each  $C_3^i$ of the $\lfloor\frac{n}{3}\rfloor$ 3-cycles of $Q$, where $i\in [1,\lfloor\frac{n}{3}\rfloor]$, we edge color its three edges by using the three elements from $B_i$.  If $r=1$, we use additional 4 new colors from $[1,d+2]\setminus (\bigcup_{i=1}^{\lfloor\frac{n}{3}\rfloor}B_i)$ to edge color the four edges of the 4-cycle of $Q$ such that each edge of the 4-cycle is assigned with a different color. If $r=2$, we use additional 5 new colors from $[1,d+2]\setminus (\bigcup_{i=1}^{\lfloor\frac{n}{3}\rfloor}B_i)$ to edge color the five edges of the 5-cycle of $Q$ such that each edge of the 5-cycle is assigned with a different color.
Denote by $\varphi_0$ this initial edge coloring of $Q$. By the definition of BIBD, $\varphi_0$ is a vertex-distinguishing edge coloring of $Q$. Denote by $C_1$ the set of colors used on edges of $Q$.

For Theorem~\ref{thm}(ii), we first define a set $C_0$ as follows.
As $d$ is an integer,  $d\ge \frac{2n}{3}$ implies that $d\ge \lceil \frac{2n}{3} \rceil =2q+r$. Thus $d+2 \ge 2q+r+2$. Let
\[C_0=\left\{\begin{array}{ll}
 \{2,2q+2\} &\text{if   $q$ is odd},\\
\{2,2q \} &\text{if $r\in[0,1]$ and $q$ is even}, \\
   \{2,2q+4\} &\text{if  $r=2$ and $q$ is even}.
           \end{array}\right.
\]

\begin{claim} The graph $G$ has a spanning subgraph $Q$ such that $Q$ is defined as follows:

\begin{itemize}
\item  If $r=0$, then $Q$ is the vertex-disjoint union of $q$ copies of $K_{1,2}$;

\item  If $r=1$, then $Q$ is the vertex-disjoint union of $q-1$ copies of $K_{1,2}$ and one copy of $K_{1,3}$;
\item If $r=2$, then $Q$ is the vertex-disjoint union of $q-1$ copies of $K_{1,2}$ and one copy of  the graph obtained from $K_{1,3}$ by subdividing exactly one edge.
\end{itemize}
\end{claim}

\begin{proof}  As $d$ is an integer,  $d\ge \frac{2n}{3}$ implies that $d\ge \lceil \frac{2n}{3} \rceil =2q+r$.
Thus the complement  graph  $\overline{G}$  of $G$ is $(n-1-d)$-regular with $n-1-d\leq n-1-2q-r = q-1$.

If $r\in [0,1]$,   applying  Theorem~\ref{thm-equitable} with $k=q$,  we get  an equitable $q$-coloring $\varphi$ of $\overline{G}$.
  Let $V_1,V_2,\ldots,V_{q}$  be the color classes of $\varphi$ in $\overline{G}$. Then  exactly $q-r$ of the $V_i$'s have size 3, while the remaining $r$ have size 4.  Thus  by deleting suitable edges from each component of $\bigcup_{i=1}^q G[V_i]$, we get a
  desired spanning subgraph $Q$ for $G$.

  If $r=2$,   applying  Theorem~\ref{thm-equitable} with $k=q+1$,  we get  an equitable $(q+1)$-coloring $\varphi$ of $\overline{G}$.
  Let $V_1,V_2,\ldots,V_{q+1}$  be the color classes of $\varphi$ in $\overline{G}$. Then  exactly $q$ of the $V_i$'s   have  size 3 and
  the rest one  has size $2$.
  Thus  exactly  $q$ of the components of $\bigcup_{i=1}^{q+1} G[V_i]$ are triangles while one is an edge.
  Let $uv$ be the component  of $\bigcup_{i=1}^{q+1} G[V_i]$ that is an edge. As $d_G(u)\ge \frac{2n}{3}$, $u$
  is adjacent in $G$ to a vertex from a triangle component of $\bigcup_{i=1}^{q+1} G[V_i]$.
  Thus $G$ contains a spanning subgraph  $Q^*$ with $q$ components such that  $q-1$ of the components are triangles
  and one is the graph obtained from a triangle and  a vertex-disjoint $K_2$ by adding exactly one edge between them.
  Let $Q$ be obtained from $Q^*$ by deleting exactly one edge from the cycle of each of its components. Then $Q$ is the desired
  spanning subgraph of $G$.
\end{proof}

We   label the vertices of $Q$ for easier reference. In $Q$,  the $q$ vertices
  that are adjacent  to vertices of degree 1 are denoted by $v_1, \ldots, v_q$, respectively.
  In particular, if $r\in [1,2]$, we let $v_q$ be the vertex that is of degree 3 in $Q$.
  For each $i\in [1,q]$,   two vertices of degree 1 that are adjacent to $v_i$ are labeled by $u_i$ and $x_i$, respectively.
  If $r=1$, then the last  so far unlabeled vertex adjacent to $v_q$ is labeled by $y_q$.
If $r=2$, then the so far unlabeled vertex that is adjacent to $v_q$ is labeled by $y_q$ and then the vertex of degree 1 that is adjacent to $y_q$
  is labeled by $z_q$.  See Figure~\ref{figure2} for an illustration of the labeling of the vertices of $Q$.

Let $C_1=\{2i+1 \mid i\in[0,q]\}$ if $r\in [0,1]$, $C_1=\{2i+1 \mid i\in[0,q]\}\cup \{2q-2\}$ if $r=2$ and $q$ is odd, and
$C_1=\{2i+1 \mid i\in[0,q]\}\cup \{2q\}$ if $r=2$ and $q$ is even.
We now give a pre-edge coloring $\varphi_{0}$ of $Q$ using colors from   $C_1$ as follows: For each $i\in[1,q]$, let
\begin{eqnarray*}
    \varphi_{0}(v_iu_i)&=&2i-1, \\
    \varphi_{0}(v_ix_i)&=&2i+1, \\
    \varphi_0(v_qy_q)&=&1 \quad \text{if $r=1$,}  \\
    \varphi_0(v_qy_q)&=&2q-2 \quad \text{and} \quad   \varphi_0(y_qz_q)=2q+1 \quad \text{if $r=2$ and $q$ is odd}, \\
    \varphi_0(v_qy_q)&=&2q  \quad \text{and} \quad   \varphi_0(y_qz_q)=2q+1 \quad \text{if $r=2$ and $q$ is even}.
\end{eqnarray*}
See Figure~\ref{figure2} for an illustration of the coloring $\varphi_0$.

\begin{figure}[!htb]
	\begin{center}
\begin{tikzpicture}

{\tikzstyle{every node}=[draw,circle,fill=black,minimum size=4pt,
                            inner sep=0pt]
 \draw (0,0) node (v1)  {}
        -- ++(60:2cm) node (u1)  {}
        -- ++(300:2cm) node (x1)   {}
        ;
         \draw (3,0) node (v2)  {}
        -- ++(60:2cm) node (u2)  {}
        -- ++(300:2cm) node (x2)   {};
          \draw (7,0) node (v_q-1)  {}
        -- ++(60:2cm) node (u_q-1)  {}
        -- ++(300:2cm) node (x_q-1)   {};
                  \draw (10,0) node (v_q)  {}
        -- ++(60:2cm) node (u_q)  {}
        -- ++(300:2cm) node (x_q)   {};
   }
   \node at (0,-0.3) {$u_1$};\node at (1,2) {$v_1$};\node at (2,-0.3) {$x_1$};
   \node at (0.1,0.7) {$1$}; \node at (1.8,0.7) {$3$};
   \node at (3,-0.3) {$u_2$};\node at (4,2) {$v_2$};\node at (5,-0.3) {$x_2$};
   \node at (3.1,0.7) {$3$}; \node at (4.8,0.7) {$5$};
   \node at (7,-0.3) {$u_{q-1}$};\node at (8,2) {$v_{q-1}$};\node at (9,-0.3) {$x_{q-1}$};
   \node[label=below:\rotatebox{60}{$2q-3$}] at (7.3,1.8) {};
   \node[label=below:\rotatebox{-60}{$2q-1$}] at (8.7,1.8) {};
   \node at (10,-0.3) {$u_{q}$};\node at (11,2) {$v_{q}$};\node at (12,-0.3) {$x_{q}$};
   \node[label=below:\rotatebox{60}{$2q-1$}] at (10.3,1.8) {};
   \node at (12,1) {$2q+1$};
\node at (6.2,0.8) {$\ldots$};

\node at (6,-1) {(a): $r=0$};

\begin{scope}[shift={(0,-4)}]

{\tikzstyle{every node}=[draw,circle,fill=black,minimum size=4pt,
                            inner sep=0pt]
 \draw (0,0) node (u1)  {}
        -- ++(60:2cm) node (v1)  {}
        -- ++(300:2cm) node (x1)   {}
        ;
         \draw (3,0) node (u2)  {}
        -- ++(60:2cm) node (v2)  {}
        -- ++(300:2cm) node (x2)   {};
          \draw (7,0) node (u_q-1)  {}
        -- ++(60:2cm) node (v_q-1)  {}
        -- ++(300:2cm) node (x_q-1)   {};
                  \draw (10,0) node (u_q)  {}
        -- ++(60:2cm) node (v_q)  {}
        -- ++(300:2cm) node (x_q)   {};
        \draw (v_q)-- ++(360:2cm) node (y_q)   {};
   }

   \node at (0,-0.3) {$u_1$};\node at (1,2) {$v_1$};\node at (2,-0.3) {$x_1$};
   \node at (0.1,0.7) {$1$}; \node at (1.8,0.7) {$3$};
   \node at (3,-0.3) {$u_2$};\node at (4,2) {$v_2$};\node at (5,-0.3) {$x_2$};
   \node at (3.1,0.7) {$3$}; \node at (4.8,0.7) {$5$};
   \node at (7,-0.3) {$u_{q-1}$};\node at (8,2) {$v_{q-1}$};\node at (9,-0.3) {$x_{q-1}$};
    \node[label=below:\rotatebox{60}{$2q-3$}] at (7.3,1.8) {};
   \node[label=below:\rotatebox{-60}{$2q-1$}] at (8.7,1.8) {};
     \node at (10,-0.3) {$u_{q}$};\node at (11,2) {$v_{q}$};\node at (12,-0.3) {$x_{q}$};\node at (13,2) {$y_{q}$};
     \node[label=below:\rotatebox{60}{$2q-1$}] at (10.3,1.8) {};
   \node at (12.2,0.7) {$2q+1$}; \node at (12,2) {$1$};
\node at (6.2,0.8) {$\ldots$};

\node at (6,-1) {(b): $r=1$};
\end{scope}

 \begin{scope}[shift={(0,-8)}]

{\tikzstyle{every node}=[draw,circle,fill=black,minimum size=4pt,
                            inner sep=0pt]
 \draw (0,0) node (u1)  {}
        -- ++(60:2cm) node (v1)  {}
        -- ++(300:2cm) node (x1)   {}
        ;
         \draw (3,0) node (u2)  {}
        -- ++(60:2cm) node (v2)  {}
        -- ++(300:2cm) node (x2)   {};
          \draw (7,0) node (u_q-1)  {}
        -- ++(60:2cm) node (v_q-1)  {}
        -- ++(300:2cm) node (x_q-1)   {};
                  \draw (10,0) node (u_q)  {}
        -- ++(60:2cm) node (v_q)  {}
        -- ++(300:2cm) node (x_q)   {};
        \draw (v_q)-- ++(360:2cm) node (y_q)   {}
        --++(300:2cm) node (z_q)   {};
   }

   \node at (0,-0.3) {$u_1$};\node at (1,2) {$v_1$};\node at (2,-0.3) {$x_1$};
   \node at (0.1,0.7) {$1$}; \node at (1.8,0.7) {$3$};
   \node at (3,-0.3) {$u_2$};\node at (4,2) {$v_2$};\node at (5,-0.3) {$x_2$};
   \node at (3.1,0.7) {$3$}; \node at (4.8,0.7) {$5$};
   \node at (7,-0.3) {$u_{q-1}$};\node at (8,2) {$v_{q-1}$};\node at (9,-0.3) {$x_{q-1}$};
   \node[label=below:\rotatebox{60}{$2q-3$}] at (7.3,1.8) {};
   \node[label=below:\rotatebox{-60}{$2q-1$}] at (8.7,1.8) {};
     \node at (10,-0.3) {$u_{q}$};\node at (11,2) {$v_{q}$};\node at  (12,-0.3) {$x_{q}$};\node at  (13,2){$y_{q}$};\node at (14,-0.3) {$z_{q}$};
   \node[label=below:\rotatebox{60}{$2q-1$}] at (10.3,1.8) {};
   \node at (12.1,0.9) {$2q+1$}; \node at (12,2) {$q^*$};\node at (14.1,0.9) {$2q+1$};
\node at (6.2,0.8) {$\ldots$};

\node at (6,-1) {(c): $r=2$, where $q^*=2q-2$ when $q$ is odd and $q^*=2q$ when $q$ is even};
 \end{scope}
\end{tikzpicture}
\end{center}
\caption{The pre-coloring $\varphi_0$ of  $Q$ for Theorem~\ref{thm}(ii)}
\label{figure2}
\end{figure}
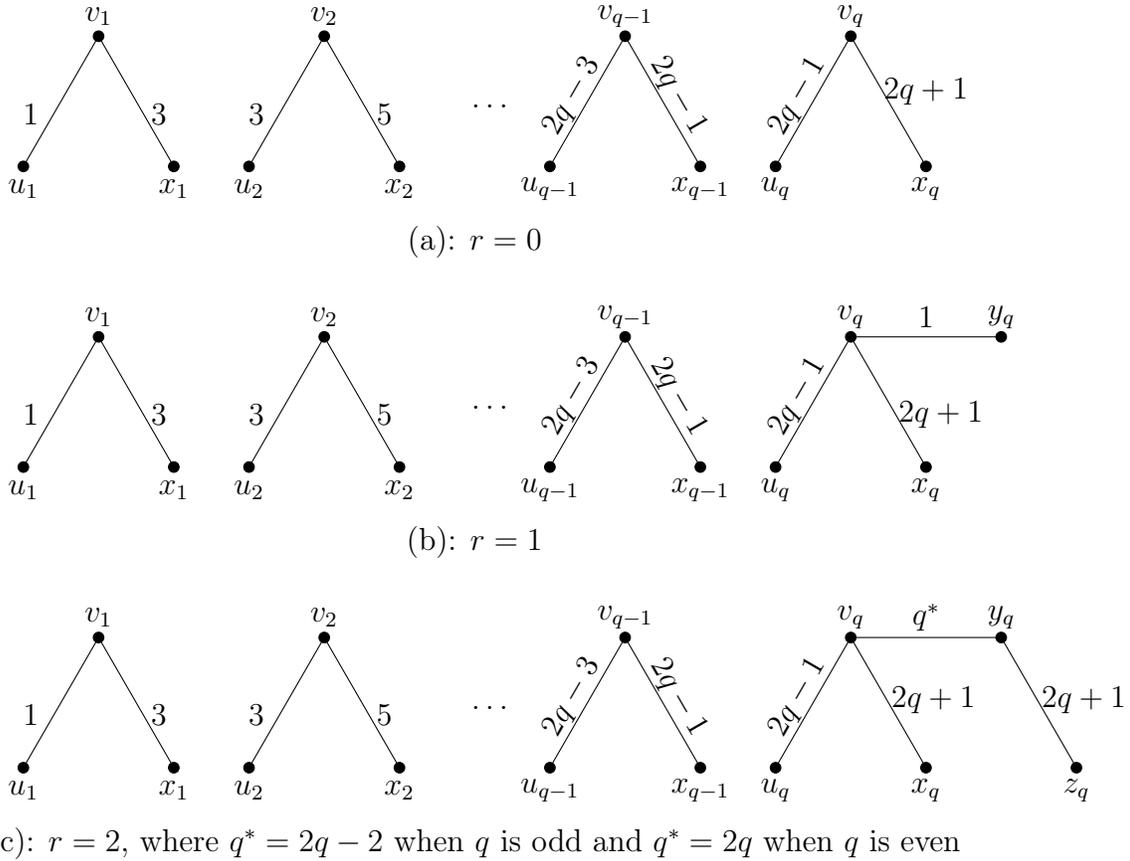

Let $m=\frac{n}{2}$.  For Theorem~\ref{thm}(i), we let $d^{*}=d+2$ and $C_0=\emptyset$;
and for Theorem~\ref{thm}(ii), we let $d^{*}=d$ and $G^{*}=G$.
We will construct an edge-$d^{*}$-coloring of $G^{*}$   in six steps. In the first five steps, we extend the precoloring $\varphi_0$ of $Q$ to $G^{*}$ using $d^{*}$ colors from $[1,d+2]\setminus C_0$. As $G^{*}$ is $d^{*}$-regular, the set of colors on edges incident
with all vertices are equal at this stage.
In the last step, for Theorem~\ref{thm}(i), we delete edges of $Q$ from $G^{*}$; and for Theorem~\ref{thm}(ii),
we  modify colors on some edges of $Q$ by using the two colors from $C_0$.
We   now provide  the detailed coloring process for the remainder of this proof, and we   use $\frac{1}{2}(1+\ve)n$
as a lower bound on $d$.

\begin{center}
    { \bf Step 1: Partition $V(G^{*})$}
\end{center}

Applying Lemma~\ref{lem:partition} with $N=\emptyset$, we partition $V(G^{*})$
into sets $A$ and $B$ such that
\begin{enumerate}[(P1)]
    \item $|A|=|B|$ and
    \item $|d_A(v)-d_B(v)| \leq m^{\frac{2}{3}}$ for each $v\in V (G^{*})$.
\end{enumerate}

Let
$$
G^{*}_A=G^{*}[A], \quad G^{*}_B=G^{*}[B], \quad G^{*}_{A,B}=G^{*}_A\cup G^{*}_B\cup Q, \quad \text{and} \quad
H=G^{*}[E_{G^{*}}(A,B)].
$$
Then for any vertex $v\in V(G^{*})$, we have that
\begin{eqnarray}
    d_A(v)&=&\frac{1}{2}((d_A(v)+d_{B}(v))+(d_A(v)-d_{B}(v)))\geq \frac{1}{2}(d^{*}-m^{\frac{2}{3}}), \nonumber\\
    \delta(G^{*}_{A,B})  &\geq& \frac{1}{2}(d^{*}-m^{\frac{2}{3}}). \label{eq2}
\end{eqnarray}
Analogously, since $\Delta(Q) \le 3$,  we have
\begin{eqnarray}
   \Delta(G^{*}_{A,B})&\leq& \frac{1}{2}(d^{*}+m^{\frac{2}{3}})+3.\label{eq3}
\end{eqnarray}

Since $d_{G}(v)=d^{*}$ for any $v\in V(G^{*})$ and $|A|=|B|$, we have  $\sum_{v\in A}d_{G^{*}}(v)=\sum_{v\in B}d_{G^{*}}(v)$. As a consequence,
\begin{equation}\label{eq4}|E(G^{*}_A)|=\frac{1}{2}\left(\sum_{v\in A}d_{G^{*}}(v)-|E_{G^{*}}(A,B)|\right)=|E(G^{*}_B)|.
\end{equation}

\begin{center}
  {\bf Step 2:   Extend $\varphi_0$ to an edge coloring $\varphi$ of  $G^{*}_{A,B}$}
\end{center}

Note that $\Delta(Q) \le 3$, and under $\varphi_0$, each color from $C_1$ is used on at most four edges of $Q$.
Thus by Lemma~\ref{lem2.2}, we can extend $\varphi_0$ to an edge coloring  $\varphi$ of $G^{*}_{A,B}$
using at most $\max\{|C_1|,\Delta(G^{*}_{A,B})+4\times 3\times 4-1=\Delta(G^{*}_{A,B})+47\}$ colors.  As $\Delta(G^{*}_{A,B}) \le \frac{1}{2}(d^{*}+m^{\frac{2}{3}})+3$
by~\eqref{eq3} and $|C_1|< \frac{1}{2}(d^{*}+m^{\frac{2}{3}})+3$, letting $k=\lceil \frac{1}{2}(d^{*}+m^{\frac{2}{3}})\rceil +50$, we extend $\varphi_0$
to an edge coloring  $\varphi$ of $G^{*}_{A,B}$ using $k$ colors from $[1,d+2]\setminus C_0$. Denote by
$C_2$ the set of the $k$ colors used by $\varphi$. As $\varphi$ is an extension of $\varphi_0$, we have
$C_1\subseteq C_2$.
Furthermore, by Lemma~\ref{PartialEC}, the edge coloring $\varphi$  of $G^{*}_{A,B}$ can be chosen such that \begin{equation}\label{eq7}||\overline{\varphi}^{-1}(\alpha)|-|\overline{\varphi}^{-1}(\beta)||\leq 17
\end{equation}
for any $\alpha,\beta\in C_2$.

As $|\overline{\varphi}(v)|=k-d_{G_{A,B}}(v) \le \frac{1}{2}(d^{*}+m^{\frac{2}{3}})+51-\frac{1}{2}(d^{*}-m^{\frac{2}{3}})\le m^{\frac{2}{3}}+51$
by the definition of $k$ and~\eqref{eq2} for each $v\in V(G^{*})$, we know that
 the average
number of vertices missing any particular color is at most
\begin{equation}
    \frac{2m(m^{2/3}+51)}{k}<\frac{2m(m^{2/3}+51)}{m/2}<5m^{2/3}-17.\label{eq8}
\end{equation}
By~\eqref{eq7}, we get
\begin{equation}\label{eq9} |\overline{\varphi}^{-1}(\alpha)|<5m^{2/3}\ \mbox{for\ any}\ \alpha\in C_2.
\end{equation}

\begin{center}
     {\bf Step 3: Extend  the $k$ color classes of $\varphi$ into $k$ perfect matchings}
\end{center}

During this step, by swapping  colors along alternating paths (paths  with edges alternating between   uncolored edges of $H$ and  edges with a given color), we will   increase the size of the $k$ color classes obtained in Step 2 until each color class is a perfect matching of $G^{*}$. The edge coloring $\varphi_0$ of $Q$ is preserved in this process.
However, we will uncolor some of the edges
of $G^{*}_A$ and $G^{*}_B$. Let $R_A$, $R_B$ be the subgraphs induced by all uncolored edges in $G^{*}_A$, $G^{*}_B$, respectively. These graphs will both initially be empty. As we proceed this step, one or two edges will be added to each of $R_A$ and $R_B$ each time we swap colors on an alternating path. We  will ensure that the following conditions are satisfied after the completion of Step 3.
\begin{enumerate}[(C1)]
    \item $|E(R_A)| = |E(R_B)| < 4m^{5/3}$;
    \item $\Delta(R_A), \Delta(R_B) < m^{5/6}$;
    \item  Each vertex of $G^{*}$ is incident with fewer than $2m^{5/6}$ colored edges of $H$.
\end{enumerate}

At any point during our process, we say that an edge $e = uv$ is \emph{good} if $e\not\in E(R_A)\cup E(R_B)$, the degree of $u$ and $v$ in both $R_A$ and $R_B$ is less than $m^{5/6}-1$. In particular, a good edge can be added to $R_A$ or $R_B$ without violating condition (C2).

By Lemma~\ref{Parity}, $|\overline{\varphi}^{-1}(\alpha)|$ is even for every color $\alpha\in C_2$.  Since $(A, B)$ is a partition of $V(G^{*})$, the quantity $|\overline{\varphi}^{-1}_A(\alpha)|-|\overline{\varphi}^{-1}_B(\alpha)|$ is also even. We can pair up
as many vertices as possible from $\overline{\varphi}^{-1}_A(\alpha)$, $\overline{\varphi}^{-1}_B(\alpha)$ and then pair up the remaining unpaired vertices from $\overline{\varphi}^{-1}_A(\alpha)$ or $\overline{\varphi}^{-1}_B(\alpha)$. We call each of these pairs a \emph{missing-common-color} pair or \emph{MCC}-pair with respect to the color $\alpha$. In fact, given any MCC-pair $(a,b)$ with respect
to some color $\alpha$, we will show that the vertices $a, b$ are connected by some alternating path $P$ (with at most 7 edges),
the path starts and also ends with an uncolored edge of $H$  and alternates between good edges (colored $\alpha$) and uncolored edges. Given such a path $P$, we will swap along $P$ by giving all the uncolored edges color $\alpha$, and uncoloring all the colored edges. After this, both $a$ and $b$ will be incident with edges colored $\alpha$, and at most three good edges will be added to $R_A\cup R_B$. Given the alternating nature of the path $P$, at most two good edges will be added to any one of $R_A, R_B$. The main difficulty in this Step 3 (and perhaps in this proof) is to show that such path $P$ exists. Before we do this, let us show that conditions (C1)-(C3) can be guaranteed at the end   of Step 3. In fact, we will only ever add good edges to $R_A$ and $R_B$, so condition (C2) will hold automatically. It remains then only to check (C1) and (C3).

Let us first consider Condition (C1). As  $\sum_{v\in V(G^{*})}|\overline{\varphi}(v)|<3m^{5/3}$ by~\eqref{eq8},
there are less than $1.5m^{5/3} < 2m^{5/3}$ MCC-pairs.
For each MCC-pair $(a, b)$ with $a, b\in V (G^{*})$, at most two edges will be added to each
of $R_A$ and $R_B$ when we exchange an alternating path connecting  $a$ and $b$. Thus there will always be less than $4m^{5/3}$ uncolored edges in each of $R_A$ and $R_B$. At the completion of Step 3, each of the $k$ color classes is a perfect matching of $G$ so both of $G^{*}_A$ and $G^{*}_B$ have the same number of colored edges. Since $|E(G^{*}_A)| = |E(G^{*}_B)|$ by~\eqref{eq4}, we have $|E(R_A)| = |E(R_B)|$.  Thus Condition (C1) will be satisfied at the end of Step 3.

We now show that Condition (C3) will also be satisfied at the end of Step 3. In the
process of Step 3, the number of newly colored edges of $H$ that are incident with a
vertex $u\in V (G^{*})$ will equal the number of our chosen alternating paths containing $u$. The number of such alternating paths of which $u$ is the first vertex will equal the number of colors missing from $u$ at the end of Step 2, so will be less than $|\overline{\varphi}(u)|<m^{2/3}+51$. The number of alternating paths in which
$u$ is not the first vertex will equal the degree of $u$ in $R_A\cup R_B$, and so will be less than $m^{5/6}$. As the edges of $Q$ were colored  before Step 1 and $\Delta(Q)\le 3$, the number of colored edges of $H$ that are incident with $u$ will be less than $|\overline{\varphi}(u)| + m^{5/6} + 3 < 2m^{5/6}$. Thus Condition (C3) will be satisfied at the end of Step 3.

Let $\alpha\in C_2$. We   describe how to find an alternating path  for an MCC-pair  with respect to $\alpha$.
Let $V_\alpha$ be the set of vertices incident in $G$ with an edge from  $Q$ that is colored by $\alpha$.
 We will   avoid  using vertices from $V_\alpha$ in all of the alternating paths connecting our MCC-pairs. For each vertex in an MCC-pair, we will see how to construct a short alternating path from it that avoids $V_{\alpha}$.

Let $v\in V(G^{*})$. Suppose that $v\in S\in \{A, B\}$, and let $T=\{A,B\}\setminus \{S\}$. Define $N_1(v)$ to be the set of all vertices in $T\setminus V_\alpha$, that are joined to $v$ by an uncolored edge, and are incident to a good edge colored $\alpha$. Let $N_2(v)$ be the set of vertices in $T\setminus V_\alpha$ that are
joined to a vertex of $N_1(v)$ by a good edge of color $\alpha$. Note that we have $|N_1(v)| = |N_2(v)|$,
but some vertices may be in $N_1(v)\cap N_2(v)$.

There are fewer than $4m^{5/3}$ edges in $R_B$ (by (C1)), so there are fewer than $8m^{5/6}$ vertices of degree at least $m^{5/6}$ in $R_B$. Each non-good edge is incident with one or two vertices of $R_B$ through the color $\alpha$. Thus the number of vertices in $B$ incident with a non-good edge colored by $\alpha$ is less than $8m^{5/6}$. In addition, by~\eqref{eq9} there are fewer than $5m^{2/3}$ vertices in $B$ that are missing color $\alpha$. As $\alpha$ is used on at most four edges of $Q$ under $\varphi_0$, we have $|V_\alpha| \le 8$.
So the number of vertices in $B$ that are not incident with a  good edge colored $\alpha$ or that are contained in $V_\alpha$
and incident with an edge colored $\alpha$ is
less than
\begin{equation}\label{eq5} 8m^{5/6} + 5m^{2/3}+8 <9m^{5/6}.
\end{equation}
By symmetry, the number of vertices in $A$ that are not incident with a  good edge colored $\alpha$ or that are contained in $V_\alpha$
and incident with an edge colored $\alpha$ is
less than $9m^{5/6}$. Therefore for any $v\in V (G^{*})$, by Condition (C3) and the fact that $\Delta(Q)\le 3$, we have
\begin{equation}\label{eq6}\begin{aligned}|N_1(v)| & \geq \frac{1}{2}(d^{*}-m^{2/3})-3- 2m^{5/6}-9m^{5/6}\\
&\geq \frac{1}{2}\left(\frac{(1+\varepsilon)}{2}\cdot 2m-m^{2/3}\right)-11m^{5/6}-2\\
&> \frac{m}{2}+\frac{\varepsilon m}{3}.
\end{aligned}
\end{equation}

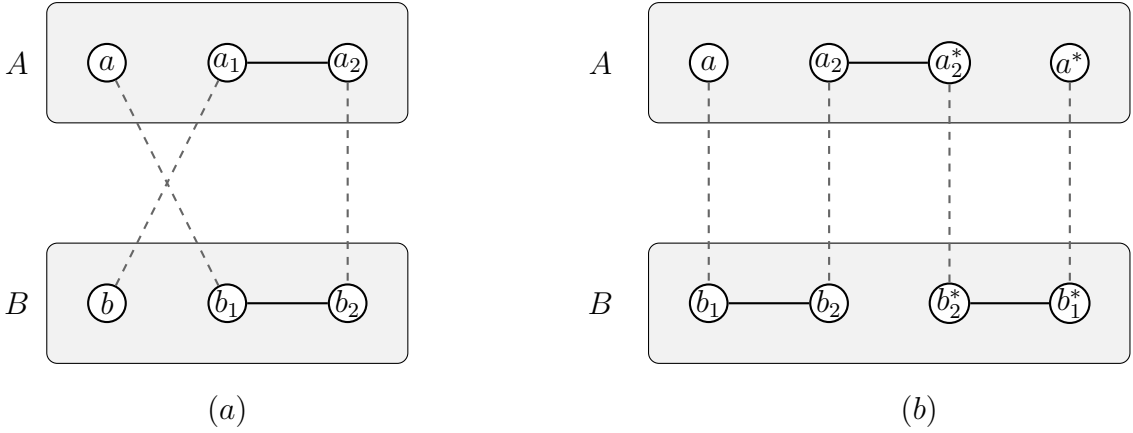
\begin{figure}[!htb]
	\begin{center}
		
		\begin{tikzpicture}[scale=0.8]

			\begin{scope}[shift={(0,0)}]
				\draw[rounded corners, fill=white!90!gray] (0, 0) rectangle (6, 2) {};
				
				\draw[rounded corners, fill=white!90!gray] (0, -4) rectangle (6, -2) {};
				
				{\tikzstyle{every node}=[draw ,circle,fill=white, minimum size=0.5cm,
					inner sep=0pt]
					\draw[black,thick](1,1) node (a)  {$a$};
					\draw[black,thick](3,1) node (a1)  {$a_1$};
					\draw[black,thick](5,1) node (a2)  {$a_2$};
					\draw[black,thick](1,-3) node (b)  {$b$};
					\draw[black,thick](3,-3) node (b1)  {$b_1$};
					\draw[black,thick](5,-3) node (b2)  {$b_2$};
				}

				\path[draw,thick,black!60!white,dashed]
				(a) edge node[name=la,pos=0.7, above] {\color{blue} } (b1)
				(a2) edge node[name=la,pos=0.7, above] {\color{blue} } (b2)
				(b) edge node[name=la,pos=0.6,above] {\color{blue}  } (a1)
				;
				
				\path[draw,thick,black]
				(a1) edge node[name=la,pos=0.7, above] {\color{blue} } (a2)
				(b1) edge node[name=la,pos=0.7, above] {\color{blue} } (b2)
				;
				
				\node at (-0.5,1) {$A$};
				\node at (-0.5,-3) {$B$};
				\node at (3,-4.8) {$(a)$};

			\end{scope}

			\begin{scope}[shift={(2,0)}]
				\draw[rounded corners, fill=white!90!gray] (8, 0) rectangle (16, 2) {};
				
				\draw[rounded corners, fill=white!90!gray] (8, -4) rectangle (16, -2) {};
				
				{\tikzstyle{every node}=[draw ,circle,fill=white, minimum size=0.5cm,
					inner sep=0pt]
					\draw[black,thick](9,1) node (c)  {$a$};
					\draw[black,thick](11,1) node (c1)  {$a_{2}$};
					\draw[black,thick](13,1) node (c2)  {$a^*_{2}$};
					\draw[black,thick](15,1) node (c3)  {$a^*$};
					\draw[black,thick](9,-3) node (d)  {$b_{1}$};
					\draw[black,thick](11,-3) node (d1)  {$b_{2}$};
					\draw[black,thick](13,-3) node (d2)  {$b^*_{2}$};
					\draw[black,thick](15,-3) node (d3)  {$b^*_{1}$};
					
				}
				\path[draw,thick,black!60!white,dashed]
				(c) edge node[name=la,pos=0.7, above] {\color{blue} } (d)
				(c1) edge node[name=la,pos=0.7, above] {\color{blue} } (d1)	
				(c3) edge node[name=la,pos=0.7, above] {\color{blue} } (d3)	
				(c2) edge node[name=la,pos=0.7, above] {\color{blue} } (d2)	
				;
				
				\path[draw,thick,black]
				(d) edge node[name=la,pos=0.7, above] {\color{blue} } (d1)
				(c2) edge node[name=la,pos=0.7, above] {\color{blue} } (c1)
				(d3) edge node[name=la,pos=0.7, above] {\color{blue} } (d2)
				;
				\node at (7.2,1) {$A$};
				\node at (7.2,-3) {$B$};
				\node at (12.5,-4.8) {$(b)$};	
			\end{scope}	
			
		\end{tikzpicture}
	\end{center}
	\caption{The alternating path $P$. Dashed lines indicate uncoloured edges, and solid
		lines indicate edges with color $i$.}
	\label{figure3}
\end{figure}

Suppose now that we have some MCC-pair $(a, b)$ with respect to $\alpha$, with $a\in A$ and
$b\in B$. By \eqref{eq6}, we have $|N_2(a)|, |N_2(b)| >\frac{m}{2}+\frac{\varepsilon m}{3}$. We choose $a_1a_2$ with color $\alpha$ such that $a_1\in N_1(b)$ and $a_2\in N_2(b)$. Now as $|N_2(a)|, |N_1(a_2)| >\frac{m}{2}+\frac{\varepsilon m}{3}$ by~\eqref{eq6}, we know that $|N_1(a_2)\cap N_2(a)| >\frac{2\varepsilon m}{3}$. We choose $b_2\in N_1(a_2)\cap N_2(a)$ such that $b_2a_2 \not\in E(Q)$.   Let $b_1\in N_1(a)$ such that $b_1b_2$ is colored by $\alpha$. Then $P = ab_1b_2a_2a_1b$ is an alternating path from $a$ to $b$ (See Figure \ref{figure3}(a)). We exchange $P$ by coloring $ab_1, b_2a_2$ and $a_1b$ with color $\alpha$ and uncoloring the edges $a_1a_2$ and $b_1b_2$. After the exchange, the color
$\alpha$ appears on edges incident with $a$ and $b$, the edge $a_1a_2$ is added to $R_A$ and the edge $b_1b_2$ is added to $R_B$.

Consider then an MCC-pair $(a, a^*)$ with respect to $\alpha$ such that $a, a^*\in A$ (the case for an MCC-pair $(b, b^{*})$ with $b, b^*\in B$ can be handled similarly). By \eqref{eq6}, we have $|N_2(a^*)| >\frac{m}{2}+\frac{\varepsilon m}{3}$. We take an edge $b_1^*b_2^*$ colored by $\alpha$ with $b_1^*\in N_1(a^*)$ and $b_2^{*}\in N_2(a^*)$.
Then again, by \eqref{eq6}, we have $|N_2(a)|, |N_2(b_2^*)| >\frac{m}{2}+\frac{\varepsilon m}{3}$. Therefore, as each vertex
$c\in N_2(b_2^*)$ satisfies $|N_1(c)| >\frac{m}{2}+\frac{\varepsilon m}{3}$, we have $|N_1(c)\cap N_2(a)| >\frac{2\varepsilon m}{3}$. We take $a_2a_2^*$ colored by $\alpha$ with $a_2^*\in N_1(b_2^*)$ such that $a_2^*b_2^*\not\in E(Q)$ and $a_2\in N_2(b_2^*)$.
Then we let $b_2\in N_1(a_2)\cap N_2(a)$ such that $a_2b_2 \not\in E(Q)$, and let $b_1$
be the vertex in $N_1(a)$ such that $b_1b_2$ is colored by $\alpha$. Now we get the alternating path $P = ab_1b_2a_2a_2^*b_2^*b_1^*a^*$ (See Figure~\ref{figure3}(b)). We exchange $P$ by coloring $ab_1, b_2a_2, a_2^*b_2^*$
and $b_1^*a^*$ with color $\alpha$ and uncoloring the edges $b_1b_2$, $b_1^*b_2^*$ and $a_2a_2^*$. After the exchange, the color $\alpha$ appears on edges incident with $a$ and $a^*$, the edges $b_1b_2$ and $b_1^*b_2^*$ are added to $R_B$ and the edge $a_2a_2^*$
is added to $R_A$. In this process, we added one edge to $R_A$ and at most two edges
to $R_B$.

We have shown above how to find alternating paths connecting all MCC-pairs, and to
modify $\varphi$ by switching on these. Each time we make such a switch, we increase by one the size of the color class $\alpha$ (and decrease by 2 the number of vertices that are missing color $\alpha$). By repeating this process, we can therefore continue until the color class $\alpha$ is a perfect matching of $G^{*}$. During this whole process, note that we did not alter any colors for edges from $E(Q)$. Thus the resulting edge coloring is still an extension of $\varphi_0$.  For simplicity, the resulting coloring is still denoted by $\varphi$.

\begin{center}
    {\bf Step 4: Color $R_A$ and $R_B$ and extend  the new color classes}
\end{center}
 Let  $\ell=\lceil m^{5/6}\rceil+1$.
In this step, we extend $\varphi$ by introducing $\ell$ new colors from the set $[1,d+2]\setminus (C_2 \cup C_0)$. By (C2), $\Delta(R_A), \Delta(R_B)<m^{5/6}$. By Vizing's Theorem, $R_A\cup R_B$ admits an edge coloring using $\ell$ colors. We let   $C_3$ be the set
of $\ell$ colors from $[1,d+2]\setminus (C_2 \cup C_0)$ used on edges of $R_A\cup R_B$.  By Theorem~\ref{lem:equa-edge-coloring}, we may assume that all these new color classes differ in size by at most one. Since $|E(R_A)|=|E(R_B)|$ by (C1),
by possibly renaming some color classes we may assume that each color
appears on the same number of edges in $R_A$ as it does in $R_B$.

By (C1), $|E(R_A)|, |E(R_B)|< 4m^{5/3}$. Since  $\ell>m^{5/6}$, each color $\alpha\in C_3$ appears on fewer than $4 m^{5/6} +1$
edges in each of $R_A$ and $R_B$. We will now color some of the edges of $H-E(Q)$ with the $\ell$ new colors so that each color class induces a perfect matching. We
 perform the following procedure for each of the $\ell$ colors in turn.

 Let $\alpha\in C_3$. We define $A_{\alpha}, B_{\alpha}$ as the sets of vertices in $A, B$, respectively, that are incident with edges colored $\alpha$.
Then, from our discussion in the previous paragraph, $|A_{\alpha}|,|B_{\alpha}| < 2 (4 m^{5/6} +1)$.   By our choices we in fact have $|A_{\alpha}|=|B_{\alpha}|$;
let $H_{\alpha}$ be the subgraph of $H-E(Q)$ obtained by
deleting the vertex sets $A_{\alpha} \cup B_{\alpha}$ and removing all colored edges. We will show that $H_{\alpha}$ has a perfect matching and we will color these matching edges with $\alpha$.

By (C3), each vertex in $V(G^{*})$ is incident with fewer than
$2m^{5/6}+\ell \le 4 m^{5/6}$
edges of $H$ that are colored. Also each vertex in $A$ has fewer than $ 2(4 m^{5/6} +1)$ edges to $B_{\alpha}$ and each vertex in $B$ has fewer than $ 2(4 m^{5/6} +1)$ edges to $A_{\alpha}$. So each vertex  $v\in V(H_{\alpha})$ is adjacent in $H_{\alpha}$ to  more than
$$
\tfrac{1}{2}\left(d^{*}-m^{2/3}\right)-4m^{5/6}-2(4 m^{5/6} +1)>\tfrac{(1+\varepsilon)m}{2}-13m^{5/6}>\frac{m}{2}
$$
vertices.
By Lemma~\ref{lem:PM}, $H_{\alpha}$
has a perfect matching $M$.
Color the edges of $M$ with the color $\alpha$. Then the color class $V_{\alpha}$ is a perfect matching of $G^{*}$.

Repeating the process above for each other color from $C_3$, we can extend each of the color classes   $\beta \in C_3$
into a perfect matching of $G^{*}$.

\begin{center}
{\bf Step 5: Color the uncolored edges}
\end{center}

Define $R$ to be the subgraph of $G^{*}$ induced on the set of uncolored edges of $G^{*}$. Then $R$ is bipartite as $G^{*}_A$ and $G^{*}_B$ are both colored in the previous steps. As each of the $k$ colors in $C_2$ and each of the $\ell$ colors in $C_3$ appears at every vertex of $G^{*}$, we have  $\Delta(R)=d^{*}-k-\ell$. By Theorem~\ref{konig},
  $R$ has an edge coloring using exactly $d^{*}-k-\ell$  colors from $[1,d+2]\setminus(C_0\cup C_2\cup C_3)$.
Denote still by $\varphi$ the resulting edge coloring of $G^{*}$.

\begin{center}
{\bf Step 6: Adjust the edge coloring $\varphi$ to be the required edge coloring}
\end{center}

For Theorem~\ref{thm}(i), as $G^{*}$ is $(d+2)$-regular and
$\varphi$  used exactly $d+2$ colors, we have
 $$s_\varphi(u)=s_\varphi(v)= [1,d+2]$$ for any distinct vertices $u,v\in V(G^{*})$. Furthermore, as $\varphi_0$ is a vertex-distinguishing edge coloring of $Q$, the edge coloring $\varphi$ of $G^{*}$ restricted on $G$ is a vd-$(d+2)$-edge coloring of $G$.

 In the following,
we modify $\varphi$ for the proof of Theorem~\ref{thm}(ii).
As $G:=G^{*}$ is $d$-regular and
$\varphi$  used exactly $d$ colors, we have
 $$\omega_\varphi(u)=\omega_\varphi(v)= \sum_{i\in [1,d+2]\setminus C_0} i$$ for any distinct vertices $u,v\in V(G)$.
 Let $s=\sum_{i\in [1,d+2]\setminus C_0} i$.

We let $\varphi_0'$ be a new edge-coloring of
$Q$ defined as follows.

If $r\in [0,1]$, for each $i\in[1,q]$, let
\begin{eqnarray*}
    \varphi_0'(v_iu_i) &=& \left\{\begin{array}{cc}
                                  2q+2 & \mbox{if}\ q\ \mbox{is\ odd}; \\
                                   2q & \mbox{if}\ q\ \mbox{is\ even},
                                 \end{array}\right.\\
   \varphi_0'(v_ix_i)& =& 2 ,   \\
 \varphi_0'(v_qy_q) &=& \varphi_0(v_qy_q)=1 \quad \text{if $r=1$}.
\end{eqnarray*}

 If $r=2$ and $q$ is odd, for each $i\in[1,q-1]$, let
 \begin{eqnarray*}
    \varphi_0'(v_iu_i) &=& 2q+2, \, \varphi_0'(v_ix_i) = 2 ,   \\
 \varphi_0'(v_qu_q)&=&2q-1,\, \varphi_0'(v_qx_q)=2, \, \varphi_0'(v_qy_q)=2q+1,  \,
      \varphi_0'(y_qz_q)=2q+2.
\end{eqnarray*}

 If $r=2$  and $q$ is even, for each $i\in[1,q-1]$, let
 \begin{eqnarray*}
    \varphi_0'(v_iu_i) &=& 2q+4, \, \varphi_0'(v_ix_i) = 2 ,   \\
 \varphi_0'(v_qu_q)&=&2q-1,\, \varphi_0'(v_qx_q)=2, \, \varphi_0'(v_qy_q)=2q+1,  \,
       \varphi_0'(y_qz_q)=2q+4.
\end{eqnarray*}

Denote by $\psi$ the resulting edge-coloring of $G$ after $\varphi_0$ is repealed by $\varphi_0'$.
It is clear that $\psi$ used exactly $d+2$ colors. We claim that $\psi$
is an edge coloring of $G$.  By the construction, $\varphi_0'$ is an edge coloring of $Q$.
As colors in $C_0$ are not used by $\varphi$, there is no other edge in $E(G)\setminus E(Q)$
that is colored by any colors of $C_0$. When $r=2$, as $\varphi(v_qx_q)=\varphi(y_qz_q)=2q+1$, and $\varphi_0'(v_qx_q)=2q+4$
and  $\varphi_0'(y_qz_q)=2$, it follows that $v_qy_q$ is the only edge incident with $v_q$ and $y_q$ that is colored by $2q+1$
under $\psi$. Thus $\psi$ is an edge-$(d+2)$-coloring of $G$.
In the following, we show that $\psi$ is sum-distinguishing.

We consider first that $q$ is odd. By the definitions of $\varphi_0$ and $\varphi'_0$, for each $i\in[1,q]$ when $r\in[0,1]$
and for each $i\in[1,q-1]$ when $r=2$, we have
\begin{eqnarray*}
    \omega_\psi(v_i)&=& s-\omega_{\varphi_0}(v_i)+\omega_{\varphi'_0}(v_i) \\
     &=& s-(2i-1+2i+1)+(2q+2+2) \\
     &=&s+2q+4-4i, \\
     \omega_\psi(u_i)&=& s-\omega_{\varphi_0}(u_i)+\omega_{\varphi'_0}(u_i) \\
     &=& s-(2i-1)+2q+2 \\
     &=& s+2q+3-2i, \\
      \omega_\psi(x_i)&=& s-\omega_{\varphi_0}(x_i)+\omega_{\varphi'_0}(x_i) \\
     &=& s-(2i+1)+2 \\
     &=& s+1-2i.
 \end{eqnarray*}
 When $r=1$, we additionally have $\omega_\psi(y_q)=s$;
 and when $r=2$, we additionally have
\begin{eqnarray*}
     \omega_\psi(v_q) &=& s-(2q+1+2q-2)+(2+2q+1)\\
     &=&s+4-2q, \\
     \omega_\psi(u_q) &=& s,\\
     \omega_\psi(x_q) &=& s -(2q+1)+2\\
     &=&s+1-2q,\\
     \omega_\psi(y_q) &=& s -(2q-2+2q+1)+(2q+1+2q+2)\\
     &=&s+4,\\
     \omega_\psi(z_q) &=& s -(2q+1)+2q+2\\
     &=&s+1.
\end{eqnarray*}

Note that  $\omega_\psi(v_i), \omega_\psi(u_i)$ and $\omega_\psi(x_i)$
are all decreasing in $i$. Thus all $\omega_\psi(v_i)$s are distinct, so are $\omega_\psi(u_i)$s and $\omega_\psi(x_i)$s.
As $\omega_\psi(v_i)$ is the sum of $s$ and an even number while each of $\omega_\psi(u_i)$ and $\omega_\psi(x_i)$
is the sum of $s$ and an odd number, we have $\omega_\psi(v_i) \ne \omega_\psi(u_j)$ and $\omega_\psi(v_i) \ne \omega_\psi(x_j)$
for any two distinct $i,j$. Also,   we have $ \omega_\psi(x_j)<s<\omega_\psi(u_i)$ for any distinct $i,j$.
Lastly, as $q$ is odd, we have $\omega_\psi(v_i) \ne s$.
Thus we have $\omega_\psi(u) \ne \omega_\psi(v)$  if distinct $u,v \in \{v_i, u_i, x_i \mid i\in [1,q]\}$ (when $r=0$),
$u,v \in \{v_i, u_i, x_i \mid i\in [1,q]\}\cup \{y_q\}$ (when $r=1$), or  $u,v \in \{v_i, u_i, x_i \mid i\in [1,q-1]\}\cup \{u_q\}$ (when $r=2$).

When $r=2$, clearly, we have $\omega_\psi(u) \ne \omega_\psi(v)$  for any distinct $u,v\in \{v_q, u_q, x_q, y_q, z_q\}$.
 Further, $\omega_\psi(x_q)<\omega_\psi(v_q)<\omega_\psi(u) $ for any $u\in V(G)\setminus \{v_q, x_q\}$;
  $\omega_\psi(y_q)=s+4 \ne \omega_\psi(u_i)$ and $\omega_\psi(y_q)=s+4 \ne \omega_\psi(x_i)$, and
  $\omega_\psi(y_q)=s+4 \ne \omega_\psi(v_i)$ as $q$ is odd.
  Lastly, we have $\omega_\psi(z_q) \ne \omega_\psi(v_i)$
  and  $\omega_\psi(x_i) <\omega_\psi(z_q) <\omega_\psi(u_i)$ for any $i\in[1,q-1]$.

  By the arguments above, we know that $\psi$ is sum-distinguishing when $q$ is odd.

We consider  now that $q$ is even and distinguish between $r\in[0,1]$ and $r=2$.

{\bf \noindent Case 1: $r\in[0,1]$}.

\medskip

  By the definitions of $\varphi_0$ and $\varphi'_0$, for each $i\in[1,q]$,  we have
\begin{eqnarray*}
    \omega_\psi(v_i)
     &=& s-(2i-1+2i+1)+(2q+2) \\
     &=&s+2q+2-4i, \\
     \omega_\psi(u_i)
     &=& s-(2i-1)+2q \\
     &=& s+2q+1-2i, \\
      \omega_\psi(x_i)
     &=& s-(2i+1)+2 \\
     &=& s+1-2i.
 \end{eqnarray*}
 When $r=1$, we additionally have $\omega_\psi(y_q)=s$.

By the same argument as in the case when $q$ is odd, we  conclude that $\psi$
is sum-distinguishing.

 {\bf \noindent Case 2: $r=2$}.

 By the definitions of $\varphi_0$ and $\varphi'_0$, for each $i\in[1,q-1]$,  we have
\begin{eqnarray*}
    \omega_\psi(v_i)
     &=& s-(2i-1+2i+1)+(2q+4+2) \\
     &=&s+2q+6-4i, \\
     \omega_\psi(u_i)
     &=& s-(2i-1)+2q+4 \\
     &=& s+2q+5-2i, \\
      \omega_\psi(x_i)
     &=& s-(2i+1)+2 \\
     &=& s+1-2i.
 \end{eqnarray*}
 In addition, we have
\begin{eqnarray*}
     \omega_\psi(v_q) &=& s-(2q+1+2q)+(2+2q+1)\\
     &=&s+2-2q, \\
     \omega_\psi(u_q) &=& s,\\
     \omega_\psi(x_q) &=& s -(2q+1)+2\\
     &=&s+1-2q,\\
     \omega_\psi(y_q) &=& s -(2q+2q+1)+(2q+1+2q+4)\\
     &=&s+4,\\
     \omega_\psi(z_q) &=& s -(2q+1)+2q+4\\
     &=&s+3.
\end{eqnarray*}

By the same argument as in the case when $q$ is odd, we  conclude that $\psi$
is sum-distinguishing.

\bibliographystyle{abbrv}
\bibliography{sde}

\end{document}